\documentclass{article}
\usepackage[utf8]{inputenc}
\usepackage{amsmath}
\usepackage{amssymb}
\usepackage{graphicx}
\usepackage{amsthm}
\usepackage[T1]{fontenc}
\usepackage{algpseudocode}
\newtheorem{theorem}{Theorem}
\newtheorem{corollary}{Corollary}

\newtheorem{prop}{Proposition}
\newtheorem{lemma}{Lemma}

\newtheorem{remark}{Remark}
\newtheorem{problem}{Problem}

\usepackage{algpseudocode}
\usepackage{algorithmicx}
\usepackage{algorithm}

\newcommand{\beq}{\begin{equation}}
\newcommand{\eeq}{\end{equation}}
\newcommand{\barr}{\left[\begin{array}}
\newcommand{\earr}{\end{array}\right]}

\newcommand{\rank}{\mbox{rank}\,}

\newcommand{\bpf}{\begin{proof}}
\newcommand{\epf}{\end{proof}}

\newcommand{\zz}{\ensuremath{\mathbb{Z}}}

\newcommand{\ff}{\ensuremath{\mathbb{F}}}

\newcommand{\quo}{\mbox{ \rm{div} }}
\newcommand{\rem}{\mbox{ \rm{mod} }}

\newcommand{\la}{\langle}
\newcommand{\ra}{\rangle}

\title{Solving Modular Linear Systems with a Constraint by parallel decomposition of the Smith form and extended Euclidean division modulo powers of primes divisors}
\author{Virendra Sule\\Professor of Practice\\Department of Computer Science and Engineering\\Indian Institute of Technology Hyderabad\\India\\virendra.sule@prjt.cse.iith.in}

\begin{document}
\maketitle

\begin{abstract}
    Integral linear systems $Ax=b$ with matrices $A$, $b$ and solutions $x$ are also required to be in integers, can be solved using invariant factors of $A$ (by computing the Smith Canonical Form of $A$). This paper explores a new problem which arises in applications, that of obtaining conditions for solving the Modular Linear System $Ax=b\rem n$ given $A,b$ in $\zz_n$ for $x$ in $\zz_n$ along with the constraint that the value of the linear function $\phi(x)=\la w,x\ra$ is coprime to $n$ for some solution $x$. In this paper we develop decomposition of the system to coprime moduli $p^{r(p)}$ which are divisors of $n$ and show how such a decomposition simplifies the computation of Smith form. This extends the well known index calculus method of computing the discrete logarithm where the moduli over which the linear system is reduced were assumed to be prime (to solve the reduced systems over prime fields) to the case when the factors of the modulus are prime powers $p^{r(p)}$. It is shown how this problem can be addressed efficiently using the invariant factors and Smith form of the augmented matrix $[A,-p^{r(p)}I]$ and conditions modulo $p$ satisfied by $w$, where $p^{r(p)}$ vary over all divisors of $n$ with $p$ prime.  
\end{abstract}


\section{Introduction}
The problem of solving the non-homogeneous modular linear system
\beq\label{ModLS}
Ax=b\rem n
\eeq
arises in many computational problems of Cryptography and Coding where, $n$ is a given modulus and $A,b$ are matrices over the modular ring $\zz_n$, such as for instance in the prime factorization of $n$ by quadratic residues and for \emph{discrete log} computation \cite{Pomm,Abhdas} over finite fields by the \emph{index calculus method}. In most such problems main practical hurdles are, the large size of the system (or the number of unknowns) along with the large size of the modulus $n$. For solving such large sized linear modular systems over large $n$ it would be highly desirable to reduce the problem by parallel computation using the Chinese remainder theorem (CRT) over moduli $p^{r(p)}$ which are coprime divisors of $n$. Surprisingly, such a reduction does not seem to have been discussed in the past literature for solving modular linear systems. A well known case of reduction to coprime factors of $n$ arises in the Baby Step Giant Step method of solving the discrete log over a cyclic group with known prime factorization of the group order $n$. This computation of the discrete log over a cyclic group of large order $n$ has been shown to be parallely decomposable over groups of order $p$, the prime divisors of $n$, well known as the Pohlig-Hellman reduction \cite{PohHel}. On the other hand, reduction of the linear system arising in the index calculus method of computing the discrete logarithm has not been known when the group order has prime power factors. The reduction in the number of variables is usually achieved by special number field sieves \cite{PadChak} in factorisation problems.  In most discrete log computations using index calculus the prime is of the form $2q+1$ where $q$ is a large prime. Hence the linear algebra step (which leads to solving a modular linear system) of solving the logarithms of the factor base by a modular linear system is performed modulo $q$ and modulo $2$ and the result is computed by CRT. Hence the reduced systems are defined over finite prime fields. But for general modulus $n$ the reduction of the linear system requires solving for moduli $p^r$ when $r>1$. In this paper we develop a reduction of the Smith form computation of an extended matrix for solving the modular system (\ref{ModLS}) over moduli $p^r$. Hence the algorithm results in parallel decomposition necessary for solving the system in $p$-adic arithmetic. 

In fact, in this paper we consider the modular system problem (\ref{ModLS}) with the further constraint that a given linear function $\phi(x)=\la w,x\ra$ is coprime to $n$. Such a problem arises in Cryptanalysis, \cite{SuleNewAtt,SuleNonLRR}, in the problem of inversion of sequences of $m$-tuples over Cartesian products of binary or small prime $p$ fields, $\ff_p^m$ or that of local inversion of non-linear maps over $\ff_p^m$ when the complexity of the sequences is considered in terms of Linear Recurrence Relations (LRR) modulo $p^m$. Previous cases of LRR as well as non-linear RRs over the finite field $\ff$ are discussed in \cite{SuleComplete,SuleNewAtt,SuleNonLRR}. For instance, when the sequence $\{s_k,k=0,1,2,\ldots\}$ has minimal linear recurrence, the unknown co-efficients of the minimal polynomial $\{x_0,\ldots,x_{(m-1)}\}$ satisfy a linear system defined by the Hankel matrix of the sequence and the invertibility of the sequence requires that $x_0\neq 0$ which is a special case of a linear function $\la w,x\ra=x_0\neq 0$ over the field of the sequence. When the LRR is considered over modulo $n$ then the corresponding condition becomes $x_0$ invertible modulo $n$. Hence the problem considered in this paper is a general problem of solving the system (\ref{ModLS}) with the constraint that a given linear function $\la w,x\ra$ is coprime to $n$.

There are two main hurdles with this problem of computing LRRs of sequences modulo $p^r$. First one is that the order of a LRR of a pseudorandom sequence over $\ff_p$ can be exponential, comparable to $p^m$. The second theoretical hurdle is that the linear system modulo $p^r$ for $r>1$ required to be solved for computing the LRR is no more over a finite field. Hence linear algebra methods (over fields) do not apply. Hence any computational method which can compute the Linear Complexity of such sequences modulo $p^r$ has to be scalable for large sizes of $p^r$. The method of decomposition proposed in this paper utilizes special version of Smith form of matrices over $p^r$ and is theoretically suitable to achieve the objective of efficient parallel computation using $p$-adic arithmetic.

\subsection{Summery of contribution}
Consider a brief summary of contributions made in this paper and the plan of presentation.
\begin{enumerate}
    \item Solution of the modular system (\ref{ModLS}) is first addressed in Section 2 using the Smith canonical form for integral matrices. An alternative direct linear algebra method of solving such a system by Gaussian reduction fails even for a simplest system $ax=b\rem n$ when $a,n$ are not coprime without first using the Euclidean division to compute the Greatest Common Divisor (GCD) of $a,n$. Solution of the constrained system (\ref{Eq:ModLSwithPhi}) is obtained in terms of conditions related to the prime divisors of $n$ and the unimodular transformations arising in the Smith form in Theorem \ref{Th:Solnwithphi}.
    \item The modular system solution can be decomposed as shown in Theorem \ref{Th:CRTversion} in terms of its residues by using the CRT by decomposing the system modulo $p^{r(p)}$ for the set of all prime divisors $p$ of $n$. However the resultant systems modulo $p^r$ are also not solvable by linear algebra over fields when there is an $r(p)>1$. Such a system then can be solved by the special Smith form for the augmented matrix $[A,-p^rI]$ as shown in Theorem 3.
    \item Next it is shown in Lemma \ref{Lem:Bezout1} that the Smith form of $[A,-p^rI]$ can be efficiently computed by the solution of the Bezout identity modulo $p^r$. This computation involves modulo $p$ divisions only by which the computation of GCD of elements of $A$ and $p^r$ is much simpler in computation compared to using Euclidean division to solve the Bezout identity. This is the crux of the contribution of this paper which shows how the direct solution of the Bezout identity modulo $p^r$ by $p$-adic arithmetic is more efficient to compute the Smith form for special augmented matrix $[A,-p^rI]$ than using the Euclidean division.
    \item The above technique leads to the methodology of reducing the Smith form computation of $[A,-nI]$ to $[A,-p^rI]$ using only modulo $p$ computation. Hence the methodology is a local prime factor decomposition analogous to the decomposition known as Pohlig-Hellman reduction of the Baby Step Giant Step method of computing the discrete logarithm over cyclic groups. Such a reduction has been missing in the literature since the publication of \cite{PohHel}. 
    \item This paper shows how use of $p$-adic arithmetic for computation of the transform of the augmented matrix $[A,-p^rI]$ to Smith form using the Bezout identity results into an efficient computation compared to the Euclidean division and Smith form computation of $[A,-nI]$ by parallel computation of Smith forms of $[A,-p^rI]$. Further, the computation of Smith form of $[A,-p^rI]$ is also accomplished efficiently for known primes $p$ using $p$-adic arithmetic. This is a computational insight which reverses the usual computation of Bezout identity for the of $\gcd{a,p^r}$ by using Euclidean division. It shows that direct $p$-adic arithmetic requires only modulo $p$ computations while Euclidean division to find this GCD requires division with respect to multiple divisors and dividends.  
    \item In the final section, Section 4, we provide a solution of the constrained linear system (\ref{LSoverK}), (\ref{phioverK}) over a field. 
    \item Appendix A, treats the case of solving a linear system modulo $q^r$ where $q$ is the power $p^d$ of a prime denoting the $p$-adic byte representation of numbers modulo $q^r$. The resulting algorithm shows that the Smith form computation modulo $q^r$ can be carried out in efficient byte operations pre-defined in the computer. 
    \item For solving modular systems, the finite field linear algebra approach is not applicable when there are prime divisors of $n$ with powers larger than unity. Hence the method proposed in this paper provides a completion of this old problem of linear algebra modulo $n$ by using $p$-adic arithmetic and reducing to moduli $p^r$.
\end{enumerate}
The above points briefly summarize the contributions of this paper. We focused mainly on the theory behind the computation and have omitted numerical illustrations, which are truly out of scope and would have required a much longer article.  This paper is based on the arxiv version \cite{SuleModLS}.

\subsection{Linear modular system with a constraint}
Condition for solvability and solutions of the modular system (\ref{ModLS}) can be obtained in terms of a nonhomogeneous \emph{Integral Linear System} of the form
\beq\label{IntegralLS}
Ax=b
\eeq
where $A,b$ are integral matrices and solutions $x$ sought are also integral. A well known approach to solving such integral systems is to reduce $A$ and $[A,b]$ to Hermit normal form \cite{Stor,WikiHNF}. A stronger normal form than the Hermit form is the Smith canonical form which is diagonal. Existence condition and all integral solutions of (\ref{IntegralLS}) can be described in terms of the Smith canonical form of $A$. In this paper we show how to find conditions to solve the modular linear system (\ref{ModLS}) with the additional constraint that $\la w,x\ra$ is coprime to $n$, using the Smith form reduction and the Chinese Reminder Theorem (CRT). This problem is stated formally as,

\begin{problem}\label{Pr:ModLSwithPhi}
Given a modulus $n$ (positive integer), matrix $A$ over $\zz_n^{k\times l}$, $b$ over $\zz_n^k$ and a linear function $\phi(x)=\la w,x\ra:\zz_n^l\rightarrow\zz_n$, find the conditions for the existence of solutions to the following system with the constraint
\beq\label{Eq:ModLSwithPhi}
Ax=b\mod n\mbox{ and $\phi(x)$ is coprime to $n$}
\eeq
Find a characterization of all solution $x\rem n$ when the conditions are satisfied and compute one of the solutions.
\end{problem}

In the following sections we show how this extended problem \ref{Pr:ModLSwithPhi} can be solved by using the Smith form reduction of an augmented matrix, the extended Euclidean division and the CRT. 

\subsection{Efficient computation by reduction of Smith form modulo $p^r$}
As briefly pointed out above in Section 1.1, Summery of Contributions, the main drivers of efficiency, in computation of the Smith form for solving the modular system (\ref{ModLS}) and subsequent checking of existence conditions for solving the constraint on $\phi(x)$ in Problem \ref{Pr:ModLSwithPhi} are
\begin{enumerate}
\item Use of CRT which allows decomposition of the problem into independent sub-problems over the moduli $p^{r(p)}$ corresponding to coprime divisors of $n$. Hence the solution has a natural parallel computational decomposition.
\item Further, the Smith form computation of the resulting augmented matrix is vastly simplified since only gcds of the form $p^k$ are needed. This is because gcd computation of a system of integers $a_1,a_2,\ldots,p^r$ for a known prime $p$ only involves discovering the largest power $p^k$ which divides all the $a_i$. Hence explicit Euclidean division by general integers is avoided.
\item Another important issue leading to the difficulty of solving modular liner system modulo $n$ arises when $n$ has divisors of the form $p^{r(p)}$ with $p$ prime and $r(p)>1$. For such cases solving the reduced linear systems $Ax=b\rem p$ over prime fields $\ff_p$ does not lead to a solution of the system modulo $n$. Hence in such cases, efficient (black box) linear algebra method of \cite{Wied} cannot be used. Such cases of $r(p)>1$ the modular systems have to be necessarily addressed using the Smith form (or invariant factors) of the matrix $[A,-p^{r(p)}I]$. Reduction of Smith form computation over moduli $p^r$ using parallel computation (and CRT) has surprisingly not appeared in the past literature.  
\end{enumerate}

Due to the above reasons, the solution of the problem is greatly simplified by parallel reduction over moduli $p^r$ and hence the computations are scalable for large number of unknown variables as well as large size of $n$. In the case when the powers $r(p)=1$ for all prime divisors $p$ of $n$ the reduction to moduli $p$ of the systems lead to independent linear systems over $\ff_p$. In such a case the individual systems can be solved very efficiently by the Wiedemann's method known as Black box Linear Algebra \cite{Wied,Vonzur}. Unfortunately, Wiedemann's method based on minimal polynomial computation of the matrix is not applicable when the matrix is not square and the modulus is $p^r$ with $r>1$ (system is not over a field).
\section{Solution of the integral linear system in terms of the Smith form}
First we consider the integral linear system (\ref{IntegralLS}) and describe the condition for the existence of solutions, as well as determine all solutions using the well known Smith form of $A$. The Smith form computation of a matrix over a Euclidean Domain such as $\zz$ and polynomials $K[X]$ over a field $K$ is a generalization of the GCD computation of set of integers and has well known algorithms \cite{Stor,Vonzur,Brad,WikiHNF}. Importance of Hermit and Smith forms of matrices have also been pointed out in the past \cite{Hung}. Hence following lemma is stated without proof.

\begin{lemma}
For an (integral) matrix $A$ in $\zz^{k\times l}$ there exist unimodular matrices $P$ in $\zz^{k\times k}$ and $Q$ in $\zz^{l\times l}$ such that 
\[
PAQ=\hat{S}
\]
where the integral matrix $\hat{S}$ is of the form
\[
\hat{S}=
\barr{ll}
S & 0\\0 & 0
\earr
\]
$S$ is an $r\times r$ diagonal matrix when $A$ has \emph{rank} $r$ (i.e. the largest order of nonzero determinant of a submatrix of $A$ is $r$)
\[
S=\mbox{diag}\,[f_1,f_2,\ldots,f_r]
\]
where $f_i$ are unique positive integers such that $f_1|f_2$, $f_2|f_3$,\ldots,$f_{r-1}|f_r$.
\end{lemma}

The matrix $\hat{S}$ is said to be the Smith canonical form of $A$ and the numbers $f_i$, $i=1,\ldots r$ are called its \emph{invariant factors}. Consider the linear system (\ref{IntegralLS}) and let $P$ in $\zz^{k\times k}$ and $Q$ in $\zz^{l\times l}$ be unimodular matrices which bring $A$ into the Smith form as above. Define $\tilde{b}=Pb$ and $\tilde{x}=Q^{-1}x$ with partitions
\[
\tilde{x}=
\barr{l}x_1\\x_2\earr
\mbox{ and }
\tilde{b}=
\barr{l}b_1\\b_2\earr
\]
where the vectors $x_1$ and $b_1$ are in $\zz^r$, $b_2$ is in $\zz^{k-r}$, $b_1=[\beta_1,\beta_2,\ldots,\beta_r]^T$.

\begin{theorem}\label{Th:SolnbySmithform}
The integral system (\ref{IntegralLS}) has a solution iff 
\begin{enumerate}
    \item $b_2=0$. 
    \item $f_i|\beta_i$ for $i=1,\ldots r$. Let $d_if_i=\beta_i$. 
\end{enumerate}    
then all solutions of (\ref{IntegralLS}) are given by
\[
x=Q
\left[
\begin{array}{l}
d_1\\\vdots\\d_r\\x_2
\end{array}
\right]
\]
for an arbitrary integral $x_2$ in $\zz^{l-r}$.
\end{theorem}

\begin{proof}
Denoting $\tilde{x}=Q^{-1}x$ the equation and transforming the equation (\ref{IntegralLS}) to the system of equations $(PAQ)(Q^{-1})x=Pb$ which is 
\[
Sx_1=b_1,\mbox{   }0=b_2
\]
in which $x_2$ does not appear. Since $S$ is diagonal with invariant factor $f_i$ on the $i$-th row, $x_1$ is integral iff $f_i|\beta_i$ for $i=1,2,\ldots,r$. Hence $x_1=[d_1,d_2,\ldots,d_r]$ and $x_2$ an arbitrary $\zz^{l-r}$ tuple gives all the solutions of the transformed equation $\tilde{x}=[x_1,x_2]^T$. Hence all solutions are 
\[
x=Q\tilde{x}
\]
as claimed.
\end{proof}

\subsection{Solution of the modular linear system}
The modular system (\ref{ModLS}) has a solution $x$ in $\zz^l$ iff there is an integral $k$-tuple $y$ in $\zz^k$ such that 
\[
Ax-Iyn=b
\]
where $I$  is an $k\times k$ identity matrix. Hence we get an equivalent integral system
\beq\label{Eq:eqIntegralsystem}
\barr{ll}
A & -nI
\earr
\barr{l}
x\\y
\earr
=b
\eeq
Then we can determine the solutions of the modular system by solving (\ref{Eq:eqIntegralsystem}) using the Smith form of $[A,-nI]$. 

\begin{lemma}\label{Lem:eqSmithform}
    There exist unimodular matrices $P$ in $\zz^{k\times k}$, $Q$ in $\zz^{(k+l)\times (k+l)}$ such that
    \[
    P[A,-nI]Q=[S,0]
    \]
    where $S$ is a $k\times k$ diagonal nonsingular matrix of invariant factors of $[A,-nI]$ and $[S,0]$ is a $k\times (k+l)$ matrix with $l$ zero columns added. Denote $\tilde{b}=Pb$. The system (\ref{Eq:eqIntegralsystem})
    has a solution iff $x_0=S^{-1}Pb$ is integral which is equivalent to $f_i|\beta_i$. For an arbitrary integral $l$ tuple $x_1$, denote $\tilde{x}=[x_0^T,x_1^T]^T$. Then 
    \beq\label{allsolns}
    \barr{l}x\\y\earr=
    Q\barr{l}x_0\\x_1\earr
    \eeq
    are all solutions of (\ref{Eq:eqIntegralsystem}).
\end{lemma}

\begin{proof}
Since the modular system (\ref{ModLS}) has a solution iff the equivalent system (\ref{Eq:eqIntegralsystem}) has an integral solution $x,y$, from the theorem (\ref{Th:SolnbySmithform}) it follows that there exist unimodular matrices $P,Q$ which transform $[A,-nI]$ to $[S,0]$ since integral rank $[A,-nI]$ is equal to $k$. The matrix $S$ is the diagonal nonsingular matrix of invariant factors of $[A,-nI]$. The integral system (\ref{Eq:eqIntegralsystem}) is then transformed to
\[
Sx_0=\tilde{b}
\]
The existence of a solution then follows iff $S$ divides $\tilde{b}$ which makes $x_0=S^{-1}\tilde{b}$ integral. Since the transformed equation does not involve $x_1$ this $\zz^l$ tuple is free. Hence all solutions $x,y$ are described by (\ref{allsolns}). 
\end{proof}

For ease of exposition consider further notations. Let the matrix $Q$ be partitioned and denoted as
\[
Q=
\barr{ll}Q_0 & Q_1\\
Q_2 & Q_3\earr
\]

Necessary and sufficient condition for solvability for the modular system (\ref{ModLS}) is thus obtained in above lemma as, the diagonals of $S$ must divide all the components of $\tilde{b}$ at the corresponding indices. If $x_0=S^{-1}\tilde{b}$, then all solutions of the modular system (\ref{ModLS}) are described by
\beq\label{SolnofMLSwithphi}
x=(Q_0x_0+Q_1x_1)\rem n
\eeq
for arbitrary $\zz^l$-tuple $x_1$. To express the conditions for uniqueness of solutions of the system (\ref{ModLS}) we need a definition of $\mbox{unimodular-}\rank_{\zz_n}A$ of a matrix $A$ over $\zz_n$. Such a rank of $A$ is $r$ if $r$ is the largest size of minors of $A$ such that if $\{a_1,a_2,\ldots,a_M\}$ is the set of all such minors then the $GCD(a_1,a_2,\ldots,a_M)$ is coprime to $n$. 

\begin{corollary}Following statements are equivalent
    \begin{enumerate}
    \item The modular system (\ref{ModLS}) has a unique solution $x$ in $\zz_n^l$. 
    \item $Q_1=0\rem n$
    \item $l\leq k$ and $\mbox{unimodular-}\rank_{\zz_n} A=l$
    \end{enumerate}
\end{corollary}

\begin{proof}
 Assume first statement holds. Consider the residue equation $Ax=b\rem p$ where $p$ is a prime divisor of $n$. If the system (\ref{ModLS}) has a unique solution then $x\rem p$ is the unique solution of the residue equation at $p$ for all $p|n$. Hence $\rank_{\zz_p}(A\rem p)=r=l$ which must be less than or equal to $k$. This implies that for every prime divisor $p$ of $n$, there exists a non-zero $l\times l$ minor of $A$ modulo $p$. Since an $l\times l$ minor of $(A\rem p)$ equals $l\times l$ minor of $A$ 
 modulo $p$, it follows that the GCD of all $l\times l$ minors of $A$ is not divisible by $p$ for all prime divisors of $n$, i.e. the GCD of $l\times l$ minors of $A$ is coprime to $n$. Hence the condition that $\mbox{unimodular-}\rank_{\zz_n}A=l$ follows. This proves that $1\Rightarrow 3$.

The expression of all solutions (\ref{SolnofMLSwithphi}) implies that $x$ is unique iff $x=Q_0x_0\rem n$ which implies that 
\[
Q_1x_1=0\rem n
\]
for all $x_1$ in $\zz_n^l$. Hence the solution is unique iff $Q_1\rem n=0$. Hence $1\Rightarrow 2$.

Assume now that condition $2$ holds. Since
\[
P[A,-nI]Q=[S,0]
\]
using the partition of $Q$ and noting that $Q_1=0$ we get
\[
P[A,-nI]
\barr{ll}Q_0 & 0\\Q_2 & Q_3\earr
=
[PAQ_0-nPQ_2,-nPQ_3]
\]
Hence (\ref{Eq:eqIntegralsystem}) can be written as
\[
[PAQ_0-nPQ_2,-nPQ_3]
\barr{l}x_0\\x_1\earr
=
Pb.
\]
Hence
\[
PAQ_0x_0=Pb\rem n
\]
Since $P$ and $Q_0$ are unimodular it follows that denoting $y_0=Q_0x_0$ the equation $Ay_0=b\rem n$ has a unique solution. Hence $2\Rightarrow 1$.

Hence equivalence of all the three statements is proved.
\end{proof}

\subsection{Solution of the constrained system}
Next consider the constraint $\phi(x)=\la w,x\ra$ required to be coprime to $n$ for some solution $x$ of (\ref{ModLS}) or equivalently that of (\ref{Eq:eqIntegralsystem}). Consider the $(l+k)$ tuple $[w^T,0^T]^T$ obtained by adding $k$ tuple of zero rows to $w$. Denote
\beq\label{w0w1def}
\tilde{w}=
\barr{l}w_0\\w_1\earr=Q^{T}\barr{l}w\\0\earr
\eeq
where $Q$ is the unimodular matrix in the Smith form decomposition of $[A,-nI]$, $w_0$ is the first $l$ rows of $\tilde{w}$ and $w_1$ is the remaining $k$ rows of $\tilde{w}$. Then we have
\[
\la w,x\ra=
\la\barr{l}w\\0\earr,Q\barr{l}x_0\\x_1\earr\ra=\la w_0,x_0\ra+\la w_1,x_1\ra
\]
Let $\{p_1,p_2,\ldots,p_r\}$ denote the set of all prime divisors of $n$. 
Since all solutions of the modular system (\ref{ModLS}) are given by (\ref{SolnofMLSwithphi}), the solutions which satisfy the constraint $<w,x>$ coprime to $n$ need to be characterized within this set of all solutions. Existence of such solutions satisfying the constraint and choice of $x_1$ to construct a solutionare described by the next Theorem.

\begin{theorem}\label{Th:Solnwithphi}
The linear system (\ref{Eq:ModLSwithPhi}) of Problem \ref{Pr:ModLSwithPhi} has a solution iff
\begin{enumerate}
    \item $x_0=S^{-1}\tilde{b}$ is integral and
    \item For all $j=1,2,\ldots r$, whenever $\la w_0,x_0\ra \rem p_j=0$, $w_1\rem p_j\neq 0$.
\end{enumerate}
All solutions $x$ are of the form (\ref{SolnofMLSwithphi}). To construct a solution $x_1$ may be chosen as follows:
\begin{enumerate}
\item $x_1=0$ if $\la w_0,x_0\ra\rem p_j\neq 0$ for all $j$.
\item If $\la w_0,x_0\ra \rem p_j=0$ for any $j$ then for all such $j$, $x_{1j}=x_1\rem p_j$ are chosen such that $\la w_1,x_{1j}\ra\rem p_j\neq 0$ and $x_1$ is obtained to satisfy the residues $x_{1j}=x_1\rem p_j$ and $x_1\rem p_j=0$ when $\la w_0,x_0\ra \rem p_j\neq 0$.
\end{enumerate}
\end{theorem}

\begin{proof}
If the conditions hold then the modular system (\ref{ModLS}) has a solution $x$ by the first condition and $\la w,x\ra=\la w_0,x_0\ra+\la w_1,x_1\ra$ for an arbitrary integral $x_1$. 

Now it remains to show how the second condition is sufficient to satisfy the constraint that $\la w,x\ra$ is coprime to $n$. This constraint is satisfied iff $\la w,x\ra\rem p_j\neq 0$ for all prime divisors $p_j$ of $n$. If for some $j$, $\la w_0,x_0\ra\rem p_j\neq 0$, then $x_1=0$ gives a solution which satisfies $\la w,x\ra$ coprime to $n$. On the other hand if $\la w_0,x_0\ra\rem p_j=0$ for all $j$ and $w_1\rem p_j\neq 0$ for some $j$, then there exists $x_{1j}$ in $\zz_{p_j}$ such that $\la w_1\rem p_j,x_{1j}\ra\neq 0$. Hence by the CRT there exists $x_1$ which satisfies the residues
\[
\begin{array}{lclll}
x_1 & = & x_{1j} & \rem p_j & \mbox{ such that $w_1\rem p_j\neq 0$}\\
 & = & 0 & \rem p_j & \mbox{ whenever $\la w_0,x_0\ra\rem p_j\neq 0$}
\end{array}
\]

Conversely, suppose that there is a solution $x$ to the system (\ref{ModLS}) such that $\la w,x\ra$ is coprime to $n$. Then clearly the first condition holds and there exists an integer $q$ such that $q\la w,x\ra=1\rem n$. Hence $\la w,x\ra$ is invertible in $\zz_n$ and it follows from the CRT that $\la w,x\ra\rem p_j\neq 0$ for all prime divisors $p_j$ of $n$. This leads to the second condition. Hence, the conditions are necessary.
\end{proof}

Theorem above shows how the Problem (\ref{Pr:ModLSwithPhi}) can be solved by using the Smith form of $[A,-nI]$ and the knowledge of prime divisors of $n$.
\section{Solution modulo $q^r$ where $q=p^d$}
Next we consider the special case of the Problem \ref{Pr:ModLSwithPhi} when the modulus is $q^r$ where $q=p^d$ and $p$ is a fixed prime. Such a case arises when the modular system is defined over strings of $r$ bytes where each byte is a string of $d$ bits in $\zz_p$ for small primes such as $2,3$. This special case of the problem of solving the modular system with constraint is stated as follows.

\begin{problem}\label{Pr:modqr}
Given $A$ in $\zz/(q^r)^{k\times l}$, $b$ in $\zz/(q^r)^k$, $w$ in $\zz/(q^r)^l$ and a linear function $\phi(x):\zz/(q^r)^l\rightarrow \zz_p$ defined by $\phi(x)=<w,x>\rem p$ determine necessary and sufficient conditions for existence of solutions $x$ in $\zz/(q^r)^l$ to the linear system
\[
Ax=b\rem (q^r)
\]
such that $\phi(x)\rem p\neq 0$. Describe construction of a solution when the conditions hold.
\end{problem}

From Theorem \ref{Th:Solnwithphi} we get the condition for solvability and the charactreization of solutions of this special case from the following theorem.

\begin{theorem}
    There exist unimodular matrices $P$, $Q$ as in Lemma \ref{Lem:eqSmithform} such that $P[A,-q^rI]Q=[S,0]$ is in Smith form where $S$ is the diagonal matrix of invariant factors. Denote $\tilde{b}=Pb$. Then the Problem \ref{Pr:ModLSwithPhi} for $n=q^r$ and $q=p^d$ has a solution iff
    \begin{enumerate}
        \item $x_0=S^{-1}\tilde{b}$ is integral.
        \item For an arbitrary $l$ tuple $x_1$ in $\zz/(q^r)$ and $\tilde{x}$ as defined in Lemma (\ref{Lem:eqSmithform}) all solutions $x$ of (\ref{Eq:eqIntegralsystem}) with $n=(q^r)$ are given by (\ref{SolnofMLSwithphi}).
        \item From the definition of $w_0$, $w_1$ in (\ref{w0w1def}), if $<w_0,x_0>\rem p=0$, $w_1\rem p\neq 0$.
        \item To construct a solution $x_1$, the component defining the solutions is chosen as $x_1=0$ when $<w_0,x_0>\rem p\neq 0$ and $<w_1,x_1>\rem\neq 0$ when $<w_0,x_0>\rem p=0$.
    \end{enumerate}
\end{theorem}

Proof follows from the steps in Lemma (\ref{Lem:eqSmithform}) and Theorem (\ref{Th:Solnwithphi}) and from the fact that $x_1$ exists when $w_1\rem p\neq 0$ such that $<w_1,x_1>\rem p\neq 0$. Since there is only one prime divisor $p$ in this special case of $n=q^r$ and $q=p^d$ the construction of $x_1$ using CRT is not necessary as in the general case. 

\subsection{Decomposition modulo $p^r$ and solution for general modulus $n$}
The modulo prime power case described above leads to the decomposition which helps in solving the general case when $n$ is given with prime factorization
\[
n=\prod_{p|n} p^{r_p}
\]
we can state the following theorem for existence and construction of a solution of Problem \ref{Pr:ModLSwithPhi} using the CRT.

\begin{theorem}\label{Th:CRTversion}
The Problem \ref{Pr:ModLSwithPhi} has a solution iff for each prime divisor $p$ of $n$ the modular system of Problem \ref{Pr:modqr} has a solution $x(p)$ for each $p$ and $q=p$, $r=r(p)$ such that $x(p)$ satisfies $<w,x(p)>\rem p\neq 0$. If $x(p)$ is a solution of the Problem \ref{Pr:modqr} for each of the prime divisors $p$ of $n$ satisfying the constraint $<w,x(p)>\rem p\neq 0$ and
\[
x=x(p)\rem p^{r(p)}\forall p|n
\]
is the solution of the residues with respect to mutually coprime factors $p^{r(p)}$, then $x$ is a solution of the Problem \ref{Pr:ModLSwithPhi}.    
\end{theorem}

\begin{proof}
    If the modular system $Ax=b\rem n$ has a solution $x$ then $x$ can be obtained uniquely from the residues $x(p)=x\rem p^{r(p)}$ which are solutions of $Ax=b\rem p^{r(p)}$ (termed as equations of residues modulo $p^{r(p)}$) and conversely if the residue equations have solutions $x(p)$ then the vector $x$ obtained by CRT with respect to coprime moduli $p^{r(p)}$ satisfies the modular system. Further the solution $x$ satisfies $<w,x>$ coprime to $n$ iff $<w,x>\rem p\neq 0$ for all prime divisors $p$. Hence every solution $x(p)$ of the residue equation must satisfy $<w,x>\rem p=<w\rem p,x(p)>\rem p\neq 0$. Conversely any solutions $x(p)$ of residue systems satisfying the constraint $<w,x(p)>\rem p\neq 0$ determine a unique solution $x$ of Problem \ref{Pr:modqr} modulo $n$. All other solutions can be obtained from such residues with respect to all prime divisors $p$.
\end{proof}

This theorem shows the decomposition of the problem from modulus $n$ to the smaller moduli $p^{r(p)}$ for prime divisors $p$ of $n$. Such a decomposition is computationally advantageous as we show next because the Smith form computation of $[A,-p^{r(p)}I]$ does not involve computation of GCDs of elements of $A$ along with $n$ but only involves division of $p^r$ to find the largest divisor. Hence solving the decomposed linear systems modulo $p^{r(p)}$ creates a notion of \emph{relative invariant factors} which are invariant factors of $[A,-p^{r(p)}I]$. These can be computed much efficiently because the prime $p$ is known (hence also all the divisors of $p^r$) and hence analogous computation of the GCD of rows and columns of $A$ required in Smith form computation of $A$ are not required. 

\subsection{Smith form reduction of $[A,-p^rI]$}
The above theorem gives solution of the special case of the modular system with constraint defined in Problem \ref{Pr:modqr} for $q=p$. Hence it is desirable to simplify the algorithm for computation of the invariant factors of $[A,-p^rI]$ since all of these are powers of the prime $p$. 

\subsubsection{Solution of the Bezout identity with a prime power dividend}
The Smith form reduction of a matrix such as $[A,-p^rI]$ with $p$ a prime by unimodular transformations, is based on computation of integral solutions $x_i,y$ to the Bezout identities such as
\beq\label{Bezoutq}
\sum_{i=1}^la_ix_i+p^ry=q
\eeq
where $a_i$ are given integers modulo $p^r$ and $q$ is the GCD of $(a_1,\ldots,a_l,p^r)$. We may call such a Bezout identity to be \emph{Bezout identity modulo $p^r$}. Since the only divisors of $p^r$ are $p^k$ for $0\leq k\leq r$, $q=p^g$ and the normalized Bezout identity (after canceling the factor $p^g$ from $a_i$ and $p^r$) is
\beq\label{Bezout}
\sum_{i=1}^la_ix_i+p^dy=1
\eeq
$d=r-g$, where we now consider $a_i$, $x_i$, $y$ to be in the modular ring $\zz/(p^d)$. Consider a solution in the simplest case
\beq\label{Bezout1}
ax+p^dy=1
\eeq
where $a$ and $p^d$ are coprime and
\beq\label{padicexpansion}
\begin{array}{lcl}
a & = & a_0+a_1p+\ldots+a_{d-1}p^{(d-1)}\\
b & = & b_0+b_1p+\ldots+b_{d-1}p^{(d-1)}\\
x & = & x_0+x_1p+\ldots+x_{d-1}p^{(d-1)}\\
y & = & y_0+y_1p+\ldots+y_{(d-1)}p^{(d-1)}
\end{array}
\eeq
are the $p$-adic representations of $a$, $x$ and $y$ in $\zz/(p^d)$ with $a_i,x_i,y_i$ in $[0,p-1]$. For any number $z$ in $\zz/(p^d)$ denote $(z)=z\rem p$ and $[z]=z\quo p$.
when $(a,p^d)$ are coprime. When $(a,p^d)$ are coprime $a$ has no divisor $p^k$ for $0\leq k\leq d$.

\begin{lemma}[Bezout identity mod $p^r$ for a single element]\label{Lem:Bezout1}
Given the single number 
\[
a=\sum_{i=0}^{d-1}a_ip^i
\]
for $i=0,\ldots,(d-1)$ with $a_0\rem p\neq 0$ is coprime to $p^d$. If $x,y$ satisfy the Bezout identity (\ref{Bezout1})
\[
axb+p^dyb=b
\]
then $xb$, $yb$ give solutions of the Bezout identity for right hand side $b$. The unique solutions $x_i$, $y_i$ of co-efficients of solutions $x,y$ of (\ref{Bezout1}) modulo $p^d$ 
are given by the following iterations 
\beq\label{IterationsforBezoutidentity}
\begin{array}{lclcl}
        t_0 & := & (a_0x_0)\\
        t_0 & = & 1\rem p\\ 
        \mbox{for $k=1,2,\dots,(d-1)$} & &\\
        t_k & := & [t_{(k-1)}]+(a_0x_k+a_1x_{(k-1)}+\ldots+a_kx_0)\\ 
        t_k & = & 0\rem p\\
        \mbox{for $j=0,1,\ldots,(d-1)$} & &\\
        t_{(d+j)} & := & [t_{(d-1+j)}]+a_1x_{(d-1+j)}+a_2x_{(d-2+j)}+\ldots+a_{(d-1)}x_{(1+j)}\\
        t_{(d+j)}+y_j & = & 0\rem p
\end{array}  
\eeq
Given the solutions $x,y$ in $\zz/(p^d)$ an integral solution to the Bezout identity (\ref{Bezout1}) is
\beq\label{IntBezout1}
ax+p^d(y-t)=1
\eeq
where $t$ exists because $ax+p^dy=1\rem p^d$
\end{lemma}

\begin{proof}
Proof follows by substitution of $p$-adic representations of $a$, $x$, $y$ in (\ref{Bezout1}) and comparing co-efficients of powers of $p$. While each co-efficient is computed modulo $p$, the carry part which is quotient by $p$ is added to the higher power of $p$ shown as $[t_{(k-1)}]$ in the next value of $t_k$ in the iteration. Note that each equation $t_k=0$ for $k>0$ has the unknown $x_k$, is of the form $a_0x_k=b_k\rem p$ where $b_k$ depends on previously computed $x_0,\ldots,x_{(k-1)}$ and only inversion modulo $p$ of $a_0$ is required for solving all $x_k$. Since $x,y$ were assumed with representations in $\zz/(p^d)$ the solution $x,y$ of the iterations (\ref{IterationsforBezoutidentity}) satisfy
\[
ax+p^dy=1\rem p^d
\]
Hence there exists an integer $t$ such that the right hand side above is $1+p^dt$. Hence a solution of the Bezout identity (\ref{Bezout1}) is given by (\ref{IntBezout1}) as claimed.
\end{proof}

\begin{remark}
    Note that to find solutions $x,y$ for (\ref{Bezout1}), it is enough to solve the equation 
    \[
    ax=1\rem p^d
    \]
    for $x$ modulo $p^d$. Once $x$ is solved, $ax=1+py$ for some $y$ which can be obtained by dividing $ax-1$ by $p^d$. Since $a$ is assumed to be in $\zz/(p^d)$ for any other integer $a$ the Bezout identity solutions $x,y$ can be found by solving $(a\rem p^d)x=1$ modulo $p^d$.
\end{remark}

If solution $x,y$ in $p^d$ of the Bezout identity (\ref{Bezout1}) are known then the matrix 
\[
Q=
\barr{lr}x & -p^d\\y & a\earr
\]
is unimodular and
\[
[a,p^d]Q=[1,0]
\]
Also $[ab,p^db]Q=[b,0]$. We can thus formally state
\begin{lemma}\label{Lem:Bezoutpr}
    If $\mbox{GCD}(a,p^r)=b$ then there exists a unimodular matrix $Q$ such that $[a,p^r]Q=b$.
\end{lemma}

\begin{remark}
    Note that the unimodular matrix $Q$ is obtained using $x,y$ the solutions of the Bezout identity (\ref{Bezout1}) for the pair $(ap^{-g},p^d)$ for $d=r-g$ which is coprime. This computation is performed as shown in iterations (\ref{IterationsforBezoutidentity}) in Lemma \ref{Lem:Bezout1} which involves division operations by $p$ to compute the co-efficients mod $p$ and the carry as quotient by $p$. The operations required to compute $x$, $y$ are 1) one mod $p$ inversion, 2) maximum of $d^2$ multiplications modulo $p$ and 3) $d$ integral divisions by $p$. The overall cost of computations of iterations (\ref{IterationsforBezoutidentity}) can be shown to be $O(l(p)d^2)$. In general the cost of computation of the extended GCD of two integers $a,b$ is known to be approximately $O(l(a)l(b))$ where $l(a)$ is the bit length of $a$ \cite{Buchmann}. If $a,b$ have prime factorizations 
    \[
    a\prod_{p_i|a}p_i^{r_i}\mbox{ and }b=\prod_{p_j|b}p_j^{r_j}
    \]
    then
    \[
    l(a)=\sum_i r_il(p_i)\mbox{ and }l(b)=\sum_jr_jl(p_j)
    \]
    Hence the cost of sequential computation of the solution of the Bezout identity is considerably reduced by considering independent representations of $a,b$ modulo powers of prime divisors of $n$. 
    For general integers solving the Bezout identity by parallel computation using CRT by reduction over every divisor $p^{r(p)}$ of the modulus $n$ is the most transparent way to reduce the cost of computation. 
\end{remark}

\begin{remark}
    Another important observation about the computation of the Bezout identity and $Q$ using the iterations (\ref{IterationsforBezoutidentity}) is that the successive Euclidean divisions of $a$ by $p^d$ are not used. The only equations to be solved at all iteration steps are of the form $a_0x=b\rem p$. On the other hand Euclidean division method of computing the Bezout identity involves several division by dividends other than $p^d$. Hence the direct computation of the Bezout identity using the $p$-adic iterations (\ref{IterationsforBezoutidentity}) is far more inexpensive.
\end{remark}

\subsubsection{Solution of Bezout identity mod $p^r$ with multiple elements}
Next we state a lemma to extend the Lemma \ref{Lem:Bezoutpr} to multiple elements

\begin{lemma}[Bezout identity for Multiple elements with GCD $p^g$]
Let $GCD(a_1,a_2,\ldots,a_N,p^r)=p^g$. Then there exists a unimodular matrix $Q$ such that $[a_1,a_2,\ldots,a_n,p^r]Q=[p^g,0,\ldots,0]$. The first column of $Q$ is $[x_1,x_2,\ldots,x_N,y]^T$ where $x_i$ and $y$ are modulo $p^r$ and satsfy
\[
\sum_{i=1}^Na_ix_i+p^ry=p^g
\]
\end{lemma}

\begin{proof}
Let $Q_1$ be a $2\time 2$ unimodular matrix as computed in Lemma \ref{Lem:Bezout1} such that $[a_N,p^r]=[p^{t_1},0]$ where $p^{t_1}$ is the GCD of $(a_N,p^r)$. By induction step assume that the GCD of $(a_{(N-1)},\ldots,a_1,p^r)$ be $p^{t_{(N-1)}}$ and let $Q_{(N-1)}$ be a $n\times n$ unimodular matrix such that
\[
[a_{(N-1)},\ldots,a_1,p^r]Q=[p^{(t_{(N-1)}},0,\ldots,0]
\]
The GCD of $(a_N,p^{t_{(N-1)}})$ is then $p^g$. Let $Q'$ be a unimodular matrix as derived in lemma 4 such that $[a_N,p^{t_{(N-1)}}Q'=[p^g,0]$. Define 
\[
Q_N=
\barr{ll}Q' & 0\\0 & I_{(n-1)}\earr
\barr{ll}1 & 0\\0 & Q_{(n-1)}\earr
\]
Then $Q_N$ is $(N+1)\times (N+1)$ unimodular matrix such that
\[
[a_N,\ldots,a_1,p^r]Q=[p^{t_{g}},0,\ldots,0]
\]
The numbers $x_i$ and $y$ from the first column of $Q$ then satisfy the Bezout identity as claimed.
\end{proof}

The construction of unimodular matrix for solving the Bezout identity modulo $p^r$ for multiple elements $(a_1,\ldots,a_N)$ and GCD $p^g$ with $p^r$ uses the computations of Lemma \ref{Lem:Bezout1} which are mod $p$ and quotients by $p$. Hence for a known prime $p$ these computations are much reduced compared to GCD computation of general $(N+1)$ elements and further solution of the general Bezout identity. In the next proposition we derive the unmodular matrix to transform the matrix $[A,-p^rI]$ to its Smith form.

\begin{prop}
    There exist unimodular matrices $P$, $Q$ such that 
    \[
    P[A,-p^rI]Q=[S_p,0]
    \]
    where $S_p$ is a diagonal matrix
    \[
    S_p=\mbox{diag}\,\{p^{r_1},p^{r_2},\ldots,p^{r_k}\}
    \]
    with $0\leq r_1\leq r_2\leq\ldots\leq r_k$.
\end{prop}

\begin{proof}
    First bring the matrix $\hat{A}=[A,-p^rI]$ by column (unimodular) transformations $\hat{A}V_1$ to the upper triangular form
    $[p^rI,A]$. Let $t\leq r$ be the largest power such that $p^t$ divides all elements in the first row of $A$. Then $p^t$ is the GCD of elements in the first row of $\hat{A}V_1$. Hence there exists a unimodular matrix $V_2$ obtained from the extended Euclidean divison representation of the first row such that $\hat{A}_2=\hat{A}V_1V_2$ has $p^t$ as $(1,1)$ element while the entire first row is zero. Repeat the same process to make the first column of $\hat{A}_2$ equal to zero by constructing a unimodular matrix $P$ and computing $P\hat{A}_2$. Repeat both steps of operation by left and right unimodular matrices till $\hat{A}$ is brought to the form
    \[
    \barr{ll}p^{k_1} & 0\\0 & A_1\earr
    \]
    Since the original matrix $\hat{A}$ has some of its $2\times 2$ minors of the form $p^{2r}$ the GCD of elements of $A_1$ must have a divisor $p^{k_2}$ where $k_2\geq k_1$. Now repeat the same steps to make $A_1$ successively into diagonal form by unimodular row and column operations. Hence after repeating such steps the matrix $[A,-p^r]$ is transformed into a diagonal form with diagonal entries $p^{k_i}$. Next add all columns from 2nd to last to the first column and then reduce the matrix to $p^{r_1}$ as $(1,1)$ entry and and another matrix $A_1$ as above. Now we can see that $0\leq r_1$ and $p^{r_1}$ divides all other powers $p^{k_i}$ for $i\geq 2$. Repeating the same process now on the resultant matrix gives the Smith form of $[A,-p^rI]$ as claimed.
\end{proof}

\begin{remark}
    The reduction to Smith form of $[A,-p^rI]$ shown in the above Proposition uses division by $p$ in all the steps to determine largest powers of $p$ as the gcd of entries of the matrices at each step. If any element in the matrix is not divisible by $p$ then the corresponding invariant factor is $1$. Using the well known result on high probability of coprimeness of random integers it follows that the invariant factors for lower indices are highly likely to be $1$.
\end{remark}

\begin{remark}
Computing the GCD and solving the Bezout identity of a set of $N$ integers and $p^r$ modulo $p^r$ has complexity equivalent to the $O(Nd^2(\log p)^2)$ where $d=r-t$ and $p^t$ is the GCD of $(a_1,\ldots,a_N,p^r)$. For a given prime $p$ the unimodular matrix $Q$ which can transform $[a_1,a_2,\ldots,a_N,p^r]$ as
\[
\left[
a_1,\ldots,a_N,p^r
\right]Q=
\left[
p^t,0,\ldots,0
\right]
\]
can be obtained in steps involving $O(Nd^2(\log p)^2)$ parallel operations. This involves only a verification of division by $p$ and modulo $p$ operations to solve equations (\ref{IterationsforBezoutidentity}). On the other hand if $p^r$ was one more general integer instead, the complexity of computing the GCD of $N+1$ general integers of length $l$ is $O((l^2(N+1))$ which increases rapidly with $l$, as observed from the lengths of general integers modulo $n$ in Remark 1.
\end{remark}

\subsection{High level algorithm to solve the modular system}
The above discussion clearly shows that the solution of the modular system (\ref{ModLS}) can be solved using the CRT by decomposition and solution modulo divisors $p^{r(p)}$ in the prime factorization of $n$. However for $r(p)>1$ such decomposition does not result into linear systems over finite fields. Hence the decomposed systems need to be solved again using the Smith form for the augmented matrix of the system $Ax=b\rem p^{r(p)}$. This step can then be carried out by the procedure described above by solving for the unimodular matrices modulo $p^{r(p)}$. Algorithm \ref{Alg:Solnmodulopr} describes the computation of solutions modulo $p^{r(p)}$.

\begin{algorithm}
\caption{Algorithm to compute the solution of the linear system modulo $p^r$}
\label{Alg:Solnmodulopr}
\begin{algorithmic}[1]
\Procedure{LSmodpr}{Solve $Ax=b\rem p^r$}
\State \textbf{Input} $A,b,p,r$.
\State\textbf{Output} Integeral Solutions $x,y$ of the system $Ax-p^ry=b$ and solution modulo $p^r$.
\State Form the augmented matrix $[A,-p^rI]$.
\State Find unimodular integral matrices $P$, $Q$ from the procedure described in Lemma 2, to determine the Smith form $[S_p,0]$ of the augmented matrix.
\State Compute 
\[
x_0=(S_p)^{-1}Pb
\]
\If{$x_0$ is not an integral} Go to \textbf{end if}
\Else
\State Return $x=Q_0x_0,y=Q_2x_0$ as one solution. See (\ref{allsolns}).
\State Choose an arbitrary integral $x_1$ in $\zz^l$. Compute arbitrary solutions $x,y$ from formula (\ref{allsolns}) dependent on $x_1$.
\State Return solutions $x,y$ and solutions $x\rem p^r$.
\EndIf
\State ``Solution does not exist".
\EndProcedure
\end{algorithmic}
\end{algorithm}

\subsection{Using the CRT to solve the modular system}
The final algorithm to solve the modular system (\ref{ModLS}) using CRT is now presented below as Algorithm \ref{Alg:Solnmodulon}.

\begin{algorithm}
\caption{Algorithm to compute the solution of the linear system modulo $n$ using CRT}
\label{Alg:Solnmodulon}
\begin{algorithmic}[1]
\Procedure{LSmodn}{Solve $Ax=b\rem n$}
\State \textbf{Input} $A,b,n$, $m$ distinct prime divisors of $n$ and their multiplicities $p,r(p)$.
\State\textbf{Output} Integeral Solutions $x,y$ of the system $Ax-ny=b$.
\State Form the augmented matrix $[A,-nI]$.
\For{For each prime divisor $p_i|n$ with multipicity $r_i$}
\State Compute $j=1,2,\ldots,m$
\[
\begin{array}{lcl}
P_j & = & \prod_{i=1,i\neq j}^{i=m}(p_i)^{r_i}\\
Q_j & = & (P_j)^{-1}\rem p_j^{r_j}
\end{array}
\]
\State Find solutions $x(i),y(i)$ of the integral system $Ax(i)-p_i^{r_i}y(i)=b$ using the procedure LS\textsc{modpr}.
\State $x(i)\leftarrow x(i)\rem p_i^{r_i}$
\If{for some $i$ a solution does not exist}
\State Go to \textbf{end if}
\Else
\State Compute $x=(\sum x(i)P_iQ_i)\rem n$
\State Return solution $x$ of (\ref{ModLS}).
\EndIf
\State ``Solution does not exist".
\EndFor
\EndProcedure
\end{algorithmic}
\end{algorithm}

Using the simplification in computation of the invariant factors of $[A,-p^rI]$ and decomposition of the problem of solving the modular system over the residues as described in Theorem \ref{Th:CRTversion}, existence and construction of solutions to Problem \ref{Pr:ModLSwithPhi} can be solved efficiently by parallel computation by using the reduction over moduli $p^{r(p)}$ which are coprime divisors of $n$. Further refinement of the computation is obtained by using efficient arithmetic modulo $q=p^d$ where $d$ is the byte length of representation of numbers in digits modulo $p$. Such a computation is described in Appendix A.
\section{Linear system with constraint over a field}
For completeness of exposition and to conclude this paper, we consider  a special case of Problem 1 over a field. This problem is likely to be only of theoretical interest but the analysis below shows that the problem can be solved purely using Linear Algebra instead of integer or modular arithmetic. Consider a fresh statement of the problem.

\begin{problem}\label{Prob:LSoverK}
    Let $K$ be a field and $A$ in $K^{m\times l}$, $b$ in $K^m$, $w$ in $K^l$ be given matrices. Determine the necessary and sufficient conditions under which the linear system
    \beq\label{LSoverK}
    Ax=b
    \eeq
    has a solution $x$ in $K^l$ such that the linear function $\phi(x)=\la w,x\ra$ satisfies
    \beq\label{phioverK}
    \phi(x)\neq 0
    \eeq
    Develop a computational procedure to determine a solution.
\end{problem}

The problem is resolved by the following 
\begin{prop}\label{Pr:LSoverK}
The linear system (\ref{LSoverK}) over $K$ has a solution $x$ over $K$ such that $\phi(x)\neq 0$ iff
\beq\label{nascK}
\begin{array}{ll}
\mbox{1} & \rank A=\rank [A|b]\\
\mbox{2} & \dim\mbox{\rm rowspan } A<\dim\mbox{\rm rowspan } 
\barr{l}A\\w^T\earr
\end{array}
\eeq
\end{prop}

\begin{proof}
    Since the linear system (\ref{LSoverK}) must have a solution, condition 1 is necessary and sufficient for existence of a solution $x$ over $K$. Let $x_0$ be one solution of $Ax=b$ over $K$. Then the set of all solutions of this system is given by
    \[
    S_0=\{x_0+\ker A\}
    \]
    Let $\ker A\neq 0$. Next requirement of the problem is that there exists a solution $x$ such that $<w,x>\neq 0$. Hence it is necessary and sufficient along with condition 1 that the system of equations
    \[
    \barr{l}A\\w^T\earr x=\barr{c}b\\0\earr
    \]
    has a solution set $S_1$ which is a strict subset of $S_0$. Let $Ax_0=b$ then $x_0$ belongs to $S_0$. If $<w,x_0>\neq 0$ then $x_0$ does not belong to the solution set $S_1$ and this shows that $S_1$ is already a strict subset of $S_0$. Hence we assume that $<w,x_0>=0$. Then the solution set of the above equations is 
    \[
    S_1=\{x_0+\ker\barr{l}A\\w^T\earr\}
    \]
    The condition required is that $S_1$ is a strict subset of $S_0$. This is equivalent to the statement
    \[
    \dim\ker\barr{l}A\\w^T\earr<\dim\ker A
    \]
    Or equivalently
    \[
    \dim\mbox{ rowspan } A<\dim\mbox{ rowspan }\barr{l}A\\w^T\earr
    \]
    which is the condition 2. Hence condition 1 and 2 together are necessary and sufficient. 
\end{proof}

\subsection{Condition when the linear equation (\ref{LSoverK}) has a unique solution}
If the linear system (\ref{LSoverK}) has a unique solution $S_0=\{x_0\}$ (or when $\ker A=0$, then the requirement that the set $S_1$ is a strict subset of $S_0$ is equivalent to $S_1=\emptyset$. Hence we observe

\begin{prop}\label{Pr:uniqueoverK}
    If the linear system (\ref{LSoverK}) over $K$ has a unique solution then the problem \ref{Prob:LSoverK} over $K$ has a solution iff
    \[
    \dim\mbox{ rowspan }
    \barr{l}A\\w^T\earr<
    \dim\mbox{ rowspan }
    \barr{ll}A & b\\w^T & 0\earr
    \]
\end{prop}

    \begin{proof} The problem has a solution iff
    \[
    \barr{l}A\\w^T\earr x=\barr{l}b\\0\earr
    \]
    has no solution (or is inconsistent) which is expressed by the nas condition
    \[
    \rank\barr{l}A\\w^T\earr<\rank
    \barr{ll}A & b\\w^T & 0\earr
    \]
    which is equivalent to the condition in the proposition.
\end{proof}

\subsection{Computation of a solution}
In the present case over a field $K$, steps to find a solution of (\ref{LSoverK}) with the constraint that $\phi(x)\neq 0$ are as follows when conditions in Proposition \ref{Pr:LSoverK} are satisfied.
\begin{enumerate}
\item Let $\ker A\neq 0$ and $\{v_1,v_2,\ldots,v_s\}$ be a basis of $\ker A$.
\item Find one solution $x_0$ of $Ax=b$. This soln exists by condition 1 of (\ref{nascK}) in Proposition \ref{Pr:LSoverK}.
\item If $\la w,x_0\ra\neq 0$, then $x_0$ satisfies the constraint. If not at least one nonzero vector $x_1$ exists in $\ker A$ such that $\la w,x_1\ra\neq 0$ due to condition 2 in (\ref{nascK}). Hence compute $c_i=\la w,v_i\ra$. If for some $i$, $\la w,v_i\ra\neq 0$, then choose $x_1=v_i$, then $x_0+x_1$ is a solution which satisfies the constraint as well as the linear system (\ref{LSoverK}).
\end{enumerate}
\section{Conclusions}
The problem of solving a modular linear system with a given linear function such that the function is coprime to the modulus is solved in terms of the Smith form of an augmented matrix of an integral linear system and evaluation of the linear function modulo prime divisors of the modulus. This procedure is then shown to be efficiently computable in parallel, by decomposing the Smith form computation over prime power divisors of the modulus and operations modulo the primes. Such a parallel reduction of the modular linear systems modulo prime powers, gives a reduction of the index calculus method of solving the discrete logarithm problem analogous to the Pohlig Hellman reduction for the discrete logarithm computation over cyclic groups when the group order has divisors which are powers of primes. For general modular linear systems the problem with the constraint on the value of a linear function on its solutions is shown to be efficiently solvable by reduction of the smith form computation for integers modulo prime powers and using the CRT to get the solution of the modular system from residues which satisfy the constraint on a linear function modulo prime powers. The arithmetic modulo prime powers is utilized in these computations for solving even the basic identities such as the Bezout identity and is the primary driver of efficiency of computation by parallelization.       

\appendix
\section{Byte Arithmetic for Smith form reduction modulo $q^r$}
To take advantage of efficient libraries for byte operations available in systems software, in the Smith form computation, the solution of the Bezout identity in iterations (\ref{IterationsforBezoutidentity}) need to be formulated modulo $q$ where $q=p^d$, $d$ is the byte length, instead of modulo $p$. This is explained below for the simplest identity for computing GCD of $(a,q^r)$ by a unimodular transformation modulo $q=p^d$.

Let $p^g=\mbox{GCD}(a,q^r)$ where $g\geq d$. We want to find $x,y$ modulo $q^r$ such that $ax+q^ry=p^g$ by using arithmetic (i.e. addition, multiplication and division) modulo $q$. Consider a $q$-adic representation of $a,x,y$ analogous to that in (\ref{padicexpansion})
\beq\label{qadicexpansion}
\begin{array}{lcl}
a & = & a_0+a_1q+\ldots+a_{r-1}q^{(r-1)}\\
x & = & x_0+x_1q+\ldots+x_{r-1}q^{(r-1)}\\
y & = & y_0+y_1q+\ldots+y_{(r-1)}q^{(r-1)}
\end{array}
\eeq
Since $a'=ap^{-g}$ and $q^{r-g}=p^{dr-g}$ are coprime, there exist integers $x,y$ such that $a'x+p^{dr-g}y=1$. Let $t=g\rem d$ and $g=md+t$, then $p^{dr-g}=q^rp^{-g}=q^{r-m}p^{-t}$. Since $t<d$, it follows that $q^{r-m}p^{-t}>q^{r-m-1}$ and hence $(a',q^{r-m-1})$ are coprime. Let $x',y'$ be a solution of
\beq\label{Bezout2}
a'x'+q^{r-m-1}y'=1
\eeq
then multiplying by $p^g$, we get
\[
\begin{array}{llcl}
 & ax'+q^rp^{-md-d}y'p^{md+t} & = & p^g\\
\Leftrightarrow & ax'+q^rp^{-d+t}y' & = & p^g
\end{array}
\]
hence $y'q^{-1}p^t$ is an integer. Hence solving for (\ref{Bezout2}) gives the solution of the Bezout identity for the $\mbox{GCD}(a,q^r)=p^g$ and subsequently a unimodular matrix $Q$ such that
\[
[a,q^r]Q=[p^g,0]
\]

\subsection{Solution modulo $q$}
The solution of (\ref{Bezout2}) can now be carried out using the arithmetic modulo $q$ because $(a,q^r)$ is coprime implies $p^k$ does not divide $a$ for any $0<k\leq r$. Hence $a\rem p\neq 0$ as well as $a\rem q\neq 0$. Hence we have unique solutions to the following iterations to solve for the co-efficients of $x',y'$ in (\ref{Bezout2}).

This is formally stated in the following lemma.

\begin{lemma}[Bezout identity mod $q^{r-(m+1)}$ for a single element]\label{Lem:Bezoutmodq}
Let $s=r-(m+1)$, given the single number 
\[
a=\sum_{i=0}^{s-1}a_iq^i
\]
for $i=0,\ldots,(s-1)$ with $a_0\rem p\neq 0$ is coprime to $q$.

The unique solutions $x_i$, $y_i$ of co-efficients of solutions $x,y$ of (\ref{Bezout1}) modulo $q$ are given by the following iterations 
\beq\label{IterationsBezoutidentitymodq}
\begin{array}{lclcl}
        t_0 & := & (a_0x_0)\\
        t_0 & = & 1\rem q\\ 
        \mbox{for $k=1,2,\dots,(t-1)$} & &\\
        t_k & := & [t_{(k-1)}]+(a_0x_k+a_1x_{(k-1)}+\ldots+a_kx_0)\\ 
        t_k & = & 0\rem q\\
        \mbox{for $j=0,1,\ldots,(t-1)$} & &\\
        t_{(t+j)} & := & [t_{(t-1+j)}]+a_1x_{(t-1+j)}+a_2x_{(t-2+j)}+\ldots+a_{(t-1)}x_{(1+j)}\\
        t_{(t+j)}+y_j & = & 0\rem q
\end{array}  
\eeq
\end{lemma}

\begin{proof}
Proof follows by substitution of $q$-adic representations of $a$, $x$, $y$ in (\ref{Bezout1}) and comparing co-efficients of powers of $q$. While each co-efficient is computed modulo $q$, the carry part which is quotient by $q$ is added to the higher power of $q$ shown as $[t_{(k-1)}]$ in the next value of $t_k$ in the iteration. Note that each equation $t_k=0$ for $k>0$ has the unknown $x_k$, is of the form $a_0x_k=b_k\rem q$ where $b_k$ depends on previously computed $x_0,\ldots,x_{(k-1)}$ and only inversion modulo $q$ of $a_0$ is required for solving all $x_k$.   
\end{proof}

\begin{remark}
The iterations (\ref{IterationsBezoutidentitymodq}) require arithmetic modulo $q$ which are byte operations. In a predefined arithmetic library the byte operations are implemented efficiently, for instance by look up tables, for addition and multiplication modulo $q$, inversion modulo $q$. One extra operation required in solving the iterations (\ref{IterationsBezoutidentitymodq}) is the quotient operation by $q$ denoted as $[t_k]$ (the carry obtained in evaluating $t_k$). Hence modulo $q^r$ computations can be carried out efficiently by byte operations which are mod $q$ to solve the Bezout identity and subsequently the Smith form by extending the unimodular matrix computations as shown in Section 4.2.2.     
\end{remark}
\end{document}